\documentclass[12pt]{amsart}
\usepackage{amsmath}
\usepackage[mathscr]{eucal}
\usepackage{amssymb}
\usepackage{latexsym}
\usepackage{amsthm}
\usepackage{pifont}
\usepackage{ascmac}
\theoremstyle{definition}
\newtheorem{thm}{Theorem}[section]

\newtheorem{prop}[thm]{Proposition}
\newtheorem{lem}[thm]{Lemma}

\newtheorem{question}[thm]{Question}

\allowdisplaybreaks
\begin{document}
\title[Cubes and Fifth Powers Sums] 
      {Remark on a Paper \\by Izadi and Baghalaghdam 
\\ about Cubes and Fifth Powers Sums} 
\author[G.~Iokibe] 
       {Gaku IOKIBE}
\address[Gaku Iokibe]
        {Department of Mathematics,
Graduate School of Science,
Osaka University,
Toyonaka, Osaka 560-0043, Japan}
\email{u325137g@alumni.osaka-u.ac.jp}

\thanks{The present paper is to appear in {\it Math. J. Okayama University}.}

\begin{abstract}
In this paper, we refine the method introduced by Izadi and Baghalaghdam to search integer solutions to the Diophantine equation $X_1^5+X_2^5+X_3^5=Y_1^3+Y_2^3+Y_3^3$. We show that the Diophantine equation has infinitely many positive solutions.
\end{abstract}
\maketitle

\section{Introduction}

In \cite{I-B}, Izadi and Baghalaghdam consider the Diophantine equation:
\begin{equation}\label{eqn:ib0}
a(X_1^{\prime 5}+X_2^{\prime 5})+\sum_{i=0}^{n} a_iX_i^5=b(Y_1^{\prime 3}+Y_2^{\prime 3})+\sum_{i=0}^{m} b_iY_i^3
\end{equation}
where $n,m\in \mathbb{N}\cup \{0\}, \ a,b\neq 0, \ a_i,b_i$ are fixed arbitrary rational numbers. They use theory of elliptic curves to find nontrivial integer solutions to (\ref{eqn:ib0}). In particular, they discuss the equation:
\begin{equation}\label{eqn:ib}
X_1^5+X_2^5+X_3^5=Y_1^3+Y_2^3+Y_3^3
\end{equation}
and obtain integer solutions, for example:
$$8^5+6^5+14^5=(-110)^3+124^3+14^3,$$
$$128122^5+(-79524)^5+48598^5=359227580^3+(-251874598)^3+107352982^3.$$
However, no positive solutions are presented in their paper \cite{I-B}. In this paper, we refine their method to find positive solutions to (\ref{eqn:ib}).

Consider the Diophantine equation (\ref{eqn:ib}). Let:
\begin{equation}\label{transform_var}
\left\{
\begin{split}
\ X_1=t+x_1, \quad &X_2=t-x_1, \quad X_3=\alpha t, \\
Y_1=t+v, \quad\  &Y_2=t-v, \qquad Y_3=\beta t.
\end{split}
\right.
\end{equation}
Then we get a quartic curve:
\begin{equation}\label{eqn:qc}
C:v^2 = \frac{2+\alpha^5}{6}t^4 + \frac{20x_1^2-2-\beta^3}{6}t^2 + \frac{5x_1^4}{3}
\end{equation}
with parameters $x_1, \ \alpha, \ \beta \in \mathbb{Q}$. If we get a rational point $(t,v)$ on $C$, we can compute a rational solution to (\ref{eqn:ib}) (see \cite{I-B}).


Once we obtain rational solutions to (\ref{eqn:ib}), we can obtain integer solutions by multiplying an appropriate value to $X_i,\ Y_i$. In the same way, in order to obtain solutions in positive integers, it suffices to search positive rational solutions to equation (\ref{eqn:ib}).


\section{Additional Requirements for Positive Solutions}

Suppose that a positive rational solution $(X_i, Y_i)_{1\leq i\leq 3}$ to (\ref{eqn:ib}) is obtained from a given point $(t,v)$ on the quartic $C$.

\begin{prop}\label{prop:1}
Let $\alpha,\beta,x_1\in\mathbb{Q}$ and 
$$
F(t)=\frac{2+\alpha^5}{6}t^4 + \frac{20x_1^2-2-\beta^3}{6}t^2 + \frac{5x_1^4}{3}.
$$
A rational point $(t,v)$ on the curve $C:v^2=F(t)$ in (\ref{eqn:qc}) produces 
a positive rational solution to (\ref{eqn:ib}) by (\ref{transform_var})
if and only if
\begin{equation} \label{1st_condition}
\alpha, \beta > 0, \quad 
0\leq F(t)<t^2,\quad 
t>|x_1|
\end{equation}
hold.
\end{prop}

\begin{proof}
If $X_i$ and $Y_i$ are positive in the solution in the form (\ref{transform_var}), 
we have
$t = (X_1 + X_2)/2 > 0$, $\alpha = 2X_3/(X_1 + X_2) > 0$ 
and $\beta = 2Y_3/(Y_1 + Y_2) > 0$. For
$(t, v)\in C$, one has that 
$0\le v^2 = F(t) < v^2 + Y_1Y_2 = t^2$. 
It follows from $x_1^2< x_1^2 + X_1X_2 = t^2$ that
$t > |x_1|$ for $t > 0$.
Conversely, suppose 
the inequalities in (\ref{1st_condition}) hold.
Then the given point $(t,v)$ on $C$ satisfies 
$v^2=F(t)<t^2$. This and (\ref{1st_condition}) 
immediately imply $X_i,Y_i>0$
in (\ref{transform_var}).
\end{proof}

\begin{prop}\label{prop:2}
Under the same assumption as Proposition \ref{prop:1}, let
\begin{equation}\label{abc_def}
a=\frac{2+\alpha^5}{6},\ b=\frac{20x_1^2-8-\beta^3}{6},\ c=\frac{5}{3}x_1^4.
\end{equation}
Then $a,b,c$ satisfy $b^2-4ac>0$ and $b<0$ if and only if there exists a real number $t$ such that $F(t)<t^2$.
\end{prop}
\begin{proof}
Let $\tilde{F}(t)=F(t)-t^2$. Since $\tilde{F}(0)=5x_1^4/3\geq 0,$ and in this case $a>0,$ 
it is easy to see that the following conditions are equivalent to each others:

(i) There exists a real number $t$ such that $F(t)<t^2$.

(ii) The equation $\tilde{F}(t)=0$ has four distinct solutions.

(iii) The quadratic equation $ax^2+bx+c=0$ has two distinct non-negative solutions.

(iv) The discriminant $D=b^2-4ac$ of the quadratic function $f(x)=ax^2+bx+c$
is positive, and the axis of the quadratic function $-b/2a$ is positive, and $f(0)\geq 0$. 

The condition (iv) holds if and only if 
``$b^2-4ac>0$ and $b<0$'',
since $a>0$ and $f(0)=c=5x_1^4/3\geq 0$.
\end{proof}

\section{Example for $X_1^5+X_2^5+X_3^5=Y_1^3+Y_2^3+Y_3^3$}

Let us first search parameters $(x_1,\alpha, \beta)$ such that 
$$
0< \alpha, \beta, \quad b<0< b^2-4ac
$$
with $a,b,c$ given by (\ref{abc_def}) 
and such that the quartic curve $C$ of (\ref{eqn:qc})
has at least one rational point.
Note that these are necessary to satisfy 
conditions of Proposition \ref{prop:1}, \ref{prop:2}.
Then, the curve $C$ is birationally equivalent 
to an elliptic curve $E$ over $\mathbb{Q}$.
If $E$ has positive rank, then $C$ has infinitely many rational points.

Let $(x_1,\alpha, \beta)=(2,1,16)$. Then the quartic:
$$C:v^2=\frac{1}{2}t^4 - \frac{2009}{3}t^2 + \frac{80}{3},$$
has a rational point $(t,v)=(44,760).$
By $T=t-44$, we transform $C$ into
$$
C' : v^2=\frac12 T^4+88 T^3+\frac{15415}{3} T^2+ \frac{334312}{3}T+ 760^2
$$
which is birationally equivalent over $\mathbb{Q}$ to 
the cubic elliptic curve (see \cite[Theorem 2.17]{W}, \cite{I-B}):
$$E:y^2+\frac{41789}{285}xy+133760y=x^3-\frac{76876021}{324900}x^2-1155200x
+\frac{2460032672}{9},$$
where:
$$T=\frac{2\cdot 760(x+\frac{15415}{3})-\frac{334312^2}{2\cdot 3^2\cdot 760}}{y},\quad v=-760+\frac{T(Tx-\frac{334312}{3})}{2\cdot 760}.
$$
Using the Sage software \cite{Sage}, 
we find that the cubic curve $E$ is an elliptic curve which has rank 2 and the generators of $E$ are:
$$P_1=\left(-\frac{1802189}{1521},\frac{5513659679}{417430}\right), \quad P_2=\left(-\frac{351379}{363},\frac{47356344241}{2276010}\right).
$$ 
We now consider the subset
$$
C_0=\left\{
(t,v)\in C\mid 0 \le F(t)<t^2\right\} \subset C
$$
whose points satisfy another condition
(\ref{1st_condition}) of 
Proposition \ref{prop:1}. 
The two quartic equations:
$$
F(t)=\frac{1}{2}t^4 - \frac{2009}{3}t^2 + \frac{80}{3}=0, \quad \tilde{F}(t)=\frac{1}{2}t^4 - \frac{2009}{3}t^2 + \frac{80}{3}-t^2=0
$$
have respectively solutions:
$$t=\pm \frac{1}{3}\sqrt{6027\pm 3\sqrt{4035601}},\quad t=\pm \frac{2}{3}\sqrt{1509\pm 3\sqrt{252979}}.$$
Let us take larger solutions as:
\begin{eqnarray*}
a_1 &=& \frac{1}{3}\sqrt{6027+3\sqrt{4035601}}\simeq 36.59635926...\\
a_2 &=& \frac{2}{3}\sqrt{1509+3\sqrt{252979}}\simeq 36.62367500...
\end{eqnarray*}
If a point $(t_0,v_0)$ on $C$ satisfies $a_1\leq t_0\leq a_2$, then 
$(t_0,v_0)$ lies on $C_0$.
We now make use of the composition law of points on the elliptic curve $E$. 
Since $E$ has positive rank, we can test infinitely many 
rational points of $E$ till finding a point $(t_0,v_0)$ on $C_0$.
We find that the rational point 
$$
Q=2P_1-P_2=\left(\tfrac{304845381192111829037}{58470412871306667},-\tfrac{4767546475726965161322288395890039}{4652843756178203561643745770}\right)
$$
on $E$ corresponds to
$$
(t_0,v_0)=\left(\tfrac{170815619844155909156204}{4664941095250009917983},
-\tfrac{690740884062625663919872925291699877683029096}{21761675422152362106175457381859866386788289}\right)
$$
on $C_0$, and creates a positive rational solution:
\begin{align*}
X_1&=\tfrac{180145502034655928992170}{4664941095250009917983}\simeq 38.61688676... \\
\vspace{2pc}
X_2&=\tfrac{161485737653655889320238}{4664941095250009917983}\simeq 34.61688676... \\
X_3&=\tfrac{170815619844155909156204}{4664941095250009917983}\simeq 36.61688676... \\
Y_1&=\tfrac{106103920658980331397442614601687483092587436}{21761675422152362106175457381859866386788289}\simeq 4.875723886... \\
Y_2&=\tfrac{1487585688784231659237188465185087238458645628}{21761675422152362106175457381859866386788289}\simeq 68.35804964... \\
Y_3&=\tfrac{2733049917506494546499264}{4664941095250009917983}\simeq 585.8701882...
\end{align*}

Next we shall prove that the Diophantine equation (\ref{eqn:ib}) has infinitely many positive solutions. The real locus of elliptic curve $E(\mathbb{R})$ can be regarded as a compact topological subspace of complex projective variety $E$. 

\begin{lem}\label{lem:1}
If the rank of elliptic curve $E$ over $\mathbb{Q}$ is positive, every point of $E(\mathbb{Q})$ is an accumulation point in $E(\mathbb{R})$.
\end{lem}

\begin{proof}
Since $E(\mathbb{R})$ is a compact topological group, and $E(\mathbb{Q})$ is an infinite 
subgroup of $E(\mathbb{R})$, there is at least one accumulation point of $E(\mathbb{Q})$ in $E(\mathbb{R})$. The group operations are homeomorphisms from $E(\mathbb{R})$ to itself. 
Therefore all points of $E(\mathbb{Q})$ are accumulation points of $E(\mathbb{R})$.
\end{proof}

\begin{thm}
The Diophantine equation (\ref{eqn:ib}) has infinitely many positive solutions.
\end{thm}
\begin{proof}
The part of $C_0$ 
has one rational point $(t_0,v_0)$
which corresponds to the above point $Q$. 
By Lemma \ref{lem:1}, the point $Q$ is an accumulation point of $E(\mathbb{Q})$ 
in $E(\mathbb{R})$, and $(t_0,v_0)$ is that of $C(\mathbb{Q})$ in $C(\mathbb{R})$.
Thus the part of $C_0$ includes infinitely many rational points. Since $2=|x_1|<a_1=36.59635926...$, 
they correspond to positive rational solutions to (\ref{eqn:ib}).
\end{proof}

\section{Example for $X_1^5+X_2^5=Y_1^3+Y_2^3+Y_3^3$}
Let $\alpha=0$. Then (\ref{eqn:ib}) gives another Diophantine equation:
\begin{equation}\label{eqn:ib2}
X_1^5+X_2^5=Y_1^3+Y_2^3+Y_3^3.
\end{equation}
In the same way, we can obtain a rational or positive rational solutions of it. For example, let $x_1=10,\ \beta = 18$. Then the quartic curve:
$$
C:v^2=\frac13 t^4-639 t^2+ \frac{50000}{3}
$$ 
has a rational point $(t,v)=(-5,30)$ and can be regarded as an elliptic curve over $\mathbb{Q}$ that has rank 2. It is birationally equivalent to:
$$
E: y^2+\frac{1867}{9}xy-400y=
x^3-\frac{3676525}{324} x^2-1200 x+ \frac{367652500}{27}.
$$
From this, we can compute positive rational solutions to (\ref{eqn:ib2}). 
For example, there is a point $Q=(x_0,y_0)$ on $E$ with
$$
x_0=\tfrac{9233921838917810856046138588468998730}{71226852166762122405616706766475947}
$$
corresponding to $(t_0,v_0)$ on $C$ with
$$
t_0=\tfrac{7869911761727476320751662986237524106650}{180965667579279848488380712753242417827}
$$
which creates the following solution to (\ref{eqn:ib2}):
\begin{align*}
X_1&=\tfrac{9679568437520274805635470113769948284920}{180965667579279848488380712753242417827},\\
X_2&=\tfrac{6060255085934677835867855858705099928380}{180965667579279848488380712753242417827},\\
Y_1&=\tfrac{2102579397586077496858869804126511993988094601307100986270258503567645177035000}{32748572842414417658282657731373155447687070419319181813277645661864847401929},\\
Y_2&=\tfrac{745788273916000738265027095213285105579870143595196644344754733646843241464100}{32748572842414417658282657731373155447687070419319181813277645661864847401929},\\
Y_3&=\tfrac{141658411711094573773529933752275433919700}{180965667579279848488380712753242417827}.
\end{align*}

The case of $\beta=0$ will be discussed briefly in 5.2 below.

\section{Parameters $(x_1,\alpha ,\beta)$ from Trivial Solutions}
\subsection{}
There are several trivial solutions; for example:
$$1^5+1^5+1^5=1^3+1^3+1^3.
$$
We call solutions to (\ref{eqn:ib}) which consist of $0,\ \pm 1$ trivial.　\hspace{-0.5pc}We are going to check some of them to search integer (or positive) solutions.

A solution to (\ref{eqn:ib}) 
may decide parameter. 
For example, when $X_i=Y_i=1$ ($i=1,2,3$),
we get $(x_1,\alpha ,\beta )=(0,1,1)$. Then:
$$C:v^2=\frac{1}{2}t^4-\frac{1}{2}t^2$$
has a singular point $(t,v)=(0,0)$　\hspace{-0.5pc}and can be parametrized 
by one parameter. 
Let us divide both sides of $C$ by $t^4$ and
substitute $s,\ w$ for $1/t,\ v/t^2$ respectively. Then:
$$C':w^2=\frac{1}{2}-\frac{1}{2}s^2$$ 
has a rational point $(s,w)=(1,0)$. Hence we can parametrize rational points on $C'$ and integer solutions to (\ref{eqn:ib}). That is to say we have:
\begin{align*}
&\left(\frac{2k^2+1}{2k^2-1}\right)^5+
\left(\frac{2k^2+1}{2k^2-1}\right)^5+
\left(\frac{2k^2+1}{2k^2-1}\right)^5
\\
&= \left(\frac{4k^4-4k^3-2k-1}{(2k^2-1)^2}\right)^3
+\left(\frac{4k^4+4k^3+2k-1}{(2k^2-1)^2} \right)^3
+\left(\frac{2k^2+1}{2k^2-1}\right)^3
\end{align*}
where $k\in \mathbb{Q}$. We can see that large enough $k$ give positive solutions to (\ref{eqn:ib}). For example:
$$\left(\frac{9}{7}\right)^5+\left(\frac{9}{7}\right)^5+\left(\frac{9}{7}\right)^5=\left(\frac{27}{49}\right)^3+\left(\frac{99}{49}\right)^3+\left(\frac{9}{7}\right)^3$$
where $k=2$. Since $X_1=X_2=X_3=Y_3$, this solution also gives positive solution to another Diophantine equation
$3X^5=Y_1^3+Y_2^3+X^3$. Moreover it satisfies $X_1+X_2+X_3=Y_1+Y_2+Y_3$ because $\alpha = \beta$.

\subsection{}
From another trivial solution:
$$
1^5+0^5+0^5=1^3+0^3+0^3,
$$
we can derive parameters $(x_1,\alpha, \beta)=(\frac{1}{2},0,0)$. Then:
$$C:v^2=\frac{1}{3}t^4+\frac{1}{2}t^2+\frac{5}{48}$$
is an elliptic curve defined over $\mathbb{Q}$ with rational point $(t,v)=(\frac{1}{2},\frac{1}{2})$. It is birationally equivalent to:
$$
E:y^2+\frac{4}{3}xy+\frac{2}{3}y=x^3+\frac{5}{9}x^2-\frac{1}{3}x-\frac{5}{27}
$$
over $\mathbb{Q}$ and has rank 1. Hence we can apply the method of Section 3 
to compute positive solutions to 
\begin{equation} \label{caseC}
X_1^5+X_2^5=Y_1^3+Y_2^3
\end{equation}
as a special case of 
(\ref{eqn:ib}) with $X_1,X_2,Y_1,Y_2>0$, $X_3=Y_3=0$ 
(where $\alpha=\beta=0$ in (\ref{transform_var})). 
For example, a point
$$
Q=\left(\tfrac{10017045137918654785}{165672066306928896},
\tfrac{29224609136538294659462738431}{67433225470590933809197056}
\right)
$$ 
on $E$
corresponding to the point
$$
(t_0,v_0)=\left(
\tfrac{2806052350871126431439}{4379016004568066987998},
\tfrac{5797926783162005502807971914786692611082209}{9587890584131638439948667971418559938024002}
\right)
$$
on $C$
creates the positive solution to (\ref{caseC}):
\begin{align*}
X_1&=\tfrac{2497780176577579962719}{2189508002284033493999},\ 
X_2=\tfrac{308272174293546468720}{2189508002284033493999},\
\\
Y_1&=\tfrac{5970900430111130674379700360675596051258385}{4793945292065819219974333985709279969012001}, \\
Y_2&=\tfrac{172973646949125171571728445888903440176176}{4793945292065819219974333985709279969012001}.
\end{align*}
\subsection{}
There exists one more parameter with $\beta = 0$, $(x_1,\alpha, \beta)=(0,0,0)$, 
which is derived from the trivial solution:
$$1^5+1^5+0^5=1^3+1^3+0^3.$$
Then the rational points on:
$$C:v^2=\frac{1}{3}t^4-\frac{1}{3}t^2$$
can be parametrized. Thus we have:
{\small
\begin{equation*}
\left(\frac{3k^2+1}{3k^2-1}\right)^5+\left(\frac{3k^2+1}{3k^2-1}\right)^5
=\left(\frac{9k^4-6k^3-2k-1}{(3k^2-1)^2}\right)^3
+\left(\frac{9k^4+6k^3+2k-1 }{(3k^2-1)^2}\right)^3,
\end{equation*}
}

\vspace{-3mm}
\noindent
where $k\in \mathbb{Q}$. For example, substituting $2$ for $k$, we have:
$$
\left(\frac{13}{11}\right)^5+\left(\frac{13}{11}\right)^5+0^5=\left(\frac{91}{121}\right)^3+\left(\frac{195}{121}\right)^3+0^3.
$$
The solutions which are obtained in these way give solutions to another Diophantine equation $2X^5=Y_1^3+Y_2^3$.
\subsection{}

It is not simple to find parameters $(x_1,\alpha ,\beta )$ that produce elliptic curves for non-trivial solutions $(X_i, Y_i)_{1\leq i\leq 3}$. In particular, the author could not find a good parameter for $\beta =0,\ \alpha \neq 0$:

\begin{question}
Find (a good method for) positive solutions to: 
\begin{equation*}
X_1^5+X_2^5+X_3^5=Y_1^3+Y_2^3.
\end{equation*} 
\end{question}

\bigskip
{\it Acknowledgement}:
The author would like to thank the referee for many valuable suggestions to
improve this article.

\end{document}